\documentclass[11pt]{article}
\usepackage{amsmath, amsfonts, amsthm, amssymb, multicol}
\usepackage{graphicx}
\usepackage{float}
\usepackage{verbatim}

\usepackage[square,sort,comma,numbers]{natbib}
\usepackage{fancyhdr}
 
\numberwithin{equation}{section}

\newtheorem{thm}{Theorem}[section]
\newtheorem{lma}[thm]{Lemma}
\newtheorem{cor}[thm]{Corollary}
\newtheorem{defn}[thm]{Definition}

\newtheorem{prop}[thm]{Proposition}
\newtheorem{conj}[thm]{Conjecture}
\newtheorem{rem}[thm]{Remark}

\newtheorem{example}[thm]{Example}

\renewcommand{\geq}{\geqslant}
\renewcommand{\leq}{\leqslant}
\renewcommand{\H}{\text{H}}

\title{ \vspace{-20mm}Distance sets, orthogonal projections, \\ and passing to weak tangents }

\author{Jonathan M. Fraser}

\begin{document}

\maketitle

\begin{abstract}
We consider the Assouad dimension analogues of two important problems in geometric measure theory. These problems are tied together by the common theme of `passing to weak tangents'.  First, we solve an analogue of \emph{Falconer's distance set problem} for Assouad dimension in the plane: if a planar set has Assouad dimension strictly greater than 1, then its distance set has Assouad dimension 1.  We also obtain partial results in higher dimensions. Second, we consider how Assouad dimension behaves under \emph{orthogonal projection}.  We extend the planar projection theorem of Fraser and Orponen to higher dimensions, provide estimates on the (Hausdorff) dimension of the exceptional set of projections, and provide a recipe for obtaining results about restricted families of projections.     We provide several illustrative examples throughout.
\end{abstract}

\date{}

\section{Introduction}

\subsection{Weak tangents and Assouad dimension}

The Assouad dimension is a fundamental notion in metric geometry, which goes back to Bouligand's 1928 paper  \cite{bouligand}.  It also played a role in Furstenberg's seminal work on galleries, which began in the 1960s, where it is referred to as the star dimension \cite{Furstenberg60s, Furstenberg}.  The notion  rose to prominence again in the 1970s through the work of Assouad which established powerful connections between the Assoaud dimension and embedding theory \cite{assouadphd}.

We begin by recalling the definition, but refer the reader to \cite{ Fraser, Luukkainen, Robinson} for more details. We consider subsets of $d$-dimensional Euclidean space $\mathbb{R}^d$ $(d \in \mathbb{N})$, but some of what we say holds in more general spaces.    For a non-empty bounded set $E \subset \mathbb{R}^d$ and $r>0$, let $N_r (E)$ be the smallest number of open sets with diameter less than or equal to $r$ required to cover $E$.  The \emph{Assouad dimension} of a non-empty set $F\subseteq \mathbb{R}^d$ is defined by
\begin{eqnarray*}
\dim_\text{A} F & = &  \inf \Bigg\{ \  s \geq 0 \  : \ \text{      $ (\exists \, C>0)$ $(\forall \, R>0)$ $(\forall \, r \in (0,R) )$ $(\forall \, x \in F )$ } \\ 
&\,& \hspace{45mm} \text{ $N_r\big( B(x,R) \cap F \big) \ \leq \ C \bigg(\frac{R}{r}\bigg)^s$ } \Bigg\}
\end{eqnarray*}
where $B(x,R)$ denotes the closed ball centred at $x$ with radius $R$. It is well-known that the Assouad dimension is always an upper bound for the Hausdorff dimension and (for bounded sets) the upper box dimension.  We write $\dim_\mathrm{H}$ for the Hausdorff dimension,  $\mathcal{H}^s$ for the $s$-dimensional Hausdorff (outer) measure, $\overline{\dim}_\mathrm{B}$ for the upper box dimension, and $\dim_\mathrm{P}$ for the packing  dimension.  For precise definitions and basic properties of these concepts, we refer the reader to \cite{falconer}.

One of the most effective ways to bound the Assouad dimension of a set from below is to use \emph{weak tangents}; an  approach  pioneered by Mackay and Tyson \cite{mackaytyson}. Weak tangents are tools for capturing the extremal local structure of a set.   Let $\mathcal{K}(\mathbb{R}^d)$ denote the set of all non-empty compact subsets of  $\mathbb{R}^d$.  This is a complete metric space when equipped with the Hausdorff metric $d_\mathcal{H}$ defined by
\[
d_\mathcal{H} (A,B) = \inf \{ \delta >0 \ : \   A \subseteq B_\delta \text{ and } B \subseteq A_\delta \}
\]
where, for any $C \in \mathcal{K}(\mathbb{R}^d)$,
\[
C_\delta \ = \ \{ x \in \mathbb{R}^d \ : \ | x- y | < \delta \text{ for some } y \in C \}
\]
denotes the open $\delta$-neighbourhood of $C$.  We will also consider the space $\mathcal{K}(X)$ for a fixed non-empty compact set $X \subseteq \mathbb{R}^d$.  This is the set of all non-empty compact subsets of  $X$ and, importantly, this is a compact subset of $ \mathcal{K}(\mathbb{R}^d)$.

\begin{defn}\label{weaktangentdef}
Let $X \in \mathcal{K}(\mathbb{R}^d)$ be a fixed reference set (usually the closed unit ball or square) and let $E, F \subseteq \mathbb{R}^d$ be closed sets with $E \subseteq X$.  Suppose there exists a sequence of similarity maps $T_k: \mathbb{R}^d \to \mathbb{R}^d$ such that $d_\mathcal{H} (E,T_k(F) \cap X ) \to 0$ as $k \to \infty$.  Then $E$   is called a \emph{weak tangent} to $F$.
\end{defn}

Recall that a similarity map $T: \mathbb{R}^d \to \mathbb{R}^d$ is a map of the form $x \mapsto c O(x) + t$ where $c >0$ is a scalar (called the \emph{similarity ratio}), $O \in \mathcal{O}(\mathbb{R},d)$ is a real orthogonal matrix and $t \in \mathbb{R}^d$ is a translation.  In most instances in this paper $O$ will be the identity matrix and $c$ will be large. The following important result of Mackay and Tyson will be used throughout the paper without being mentioned explicitly.

\begin{prop}{ \em \cite[Proposition 6.1.5]{mackaytyson}.}\label{weaktangent}
Let $F \subseteq \mathbb{R}^d$ be closed, $E \subseteq \mathbb{R}^d$ be compact, and suppose $E$ is a weak tangent to $F$.   Then $\dim_\mathrm{A} F \geq \dim_\mathrm{A} E$.
\end{prop}

It turns out that one can actually achieve the Assouad dimension as the Hausdorff dimension of a weak tangent.  This beautiful fact, which is essentially due to Furstenberg, will be the key technical tool in this paper and will  allow us to pass various geometrical problems to weak tangents. This approach provides a mechanism for finding `Assouad dimension analogues' of known results concerning  the Hausdorff dimension (and other dimensions) as well as proving stronger results, not known for other dimensions.  Thus we will obtain some fundamental results about the Assouad dimension via simple direct arguments. 

\begin{thm}[Furstenberg, K\"aenm\"aki-Ojala-Rossi] \label{goodweaktangent}
Let $F \subseteq \mathbb{R}^d$ be closed and non-empty with  $\dim_\mathrm{A} F = s \in [0,d]$.  Then there exists a compact set $E \subseteq \mathbb{R}^d$ with $\mathcal{H}^s(E) > 0$ which is a weak tangent to $F$.  In particular, $\dim_\mathrm{H} E = \dim_\mathrm{A} E =  s$.
\end{thm}

This result essentially  appears in \cite[Section 5]{Furstenberg}, without the conclusion that the weak tangent has positive $\mathcal{H}^s$ measure and using  different terminology: weak tangents are replaced by `microsets' and Assouad dimension by `star dimension', $\dim^*$.    In fact the ideas go back further to Furstenberg's work in the 1960s, see \cite{Furstenberg60s}.  For an explicit explanation of the equivalence of Furstenberg's setting and our setting, see \cite[Corollary 5.2]{kaenmakiassouad1}.  The theorem also appears using our terminology in \cite[Propositions 5.7-5.8]{kaenmakiassouad2}.  K\"aenm\"aki-Ojala-Rossi give a direct and transparent proof that if $d=2$, then there exists a weak tangent $E$ with $\dim_\mathrm{H} E = s$  and an inspection of their proof in fact yields  $\mathcal{H}^s(E) > 0$.  Moreover, their proof easily extends to the case of general $d$.  The argument presented in the proof of \cite[Propositions 5.7]{kaenmakiassouad2} is a development of an argument of Bishop-Peres \cite[Lemma 2.4.4]{bishopperes} which showed that in the case $d=1$ there exists a weak tangent with Hausdorff dimension equal to the upper box dimension of the original set.   Also, \cite[Theorem 2.4]{patches} contains the above result in the case when $s=d$, i.e. for sets with full Assouad dimension.  In this case one can actually find a weak tangent with  non-empty interior (which is stronger than positive $\mathcal{H}^d$ measure). 

The term \emph{weak tangent} is used in place of \emph{tangent} for two reasons: there is no special point in $F$ at which we `zoom-in', and the similarity ratios $c$ associated with the maps $T_k$ need not `blow-up'.   We reserve the word `tangent' for weak tangents obtained by actually zooming in at a fixed point in the set.  More precisely,  we call $E$ from the above definition a \emph{tangent} to $F$ if the following additional assumptions are satisfied: $X=B(0,1)$ (the closed unit ball);  the sequence of similarity maps $T_k$ satisfies $T_k(x_0) = 0$ for all $k$ and a fixed point $x_0 \in F$ (i.e. we zoom-in to the point $x_0 \in F$); and $c_k \nearrow \infty$ as $k \to \infty$ (i.e. we actually zoom-in).  We end this section with the observation that weak tangents are really needed to effectively study the Assouad dimension and that the more direct notion of \emph{tangent}  is not sufficient. 
\begin{example} \label{notangentexample}
There exists a compact set $F \subset [0,1]$ with $\dim_\mathrm{A} F =  1$, but such that all tangents to $F$ have Assouad dimension equal to 0.
\end{example}

We will provide the details of this example in Section \ref{notangentexampleproof}.

\begin{rem}
We recently learned that a similar example appears in \cite[Example 2.20]{rajala}.
\end{rem}

\section{Results}

\subsection{Distance sets}

Given a set $F \subseteq \mathbb{R}^d$, the distance set of $F$ is defined by
\[
D(F) = \{ |x-y| \  :  \ x,y \in F\}.
\]
 The distance set problem, which stems from the seminal paper of Falconer \cite{distancesets}, is to relate the dimensions of $F$ with the dimensions of $D(F)$.  \emph{Falconer's distance set conjecture} refers to several related conjectures, one version of which is as follows:
\begin{conj}[Falconer] \label{falconerconjecture}
Let $d \geq2$ and $F \subseteq \mathbb{R}^d$ be an analytic set.  If  $\dim_\mathrm{H} F  > d/2$, then  $D(F)$ has positive Lebesgue measure (or even non-empty interior).  In particular, it also has full Hausdorff dimension.
\end{conj}
There are numerous partial results in support of this conjecture, but the problem is still wide open, even in the plane.  One may replace the Hausdorff dimension with a different notion of dimension, such as the packing, box-counting or Assouad, and obtain a different conjecture. 
\begin{conj}  \label{distanceconj}
Let $\dim$ denote the Hausdorff, packing, upper or lower box or Assouad dimension, let $d \geq2$ and $F \subseteq \mathbb{R}^d$ be a bounded analytic set.  If $\dim F  > d/2$, then  $\dim D(F) = 1$.
\end{conj}
To the best of our knowledge, the different versions of this conjecture are all open for all values of $d \geq 2$. See \cite[Conjecture 4.5]{MattilaFourier} and the subsequent  discussion related to the Hausdorff dimension version.  An important recent result of Orponen \cite{orponenAD} is that an AD-regular set in the plane with dimension at least 1 has a distance set with full \emph{packing} dimension.  Shmerkin later proved that a Borel set in the plane with equal Hausdorff and packing dimensions strictly greater than 1 has a distance set with full Hausdorff dimension, see \cite{pablo, pablo2}.   Since these results require assumptions about more than one dimension, they do not resolve Conjecture \ref{distanceconj} for any particular dimension in isolation. Our main result resolves Conjecture \ref{distanceconj} for the Assouad dimension in the case $d=2$. Note that the analyticity assumption is not needed.

\begin{thm} \label{distanceset}
If $F \subseteq \mathbb{R}^2$ is any set with $\dim_\mathrm{A} F >1$, then  $\dim_\text{\emph{A}} D( F)  = 1.$
\end{thm}

We will prove Theorem \ref{distanceset} in Section \ref{distancesetproof}.   Clearly, $\dim_\mathrm{A} F >1$ does not guarantee that $D(F)$ has non-empty interior, but our result combined with \cite[Theorem 2.4]{patches} shows that it \emph{does} guarantee the existence of a weak tangent to $D(F)$ with non-empty interior. Moreover, in the setting of Conjecture \ref{falconerconjecture} we obtain the following result on the level of weak tangents.

\begin{cor}
If $F \subseteq \mathbb{R}^2$ is any set with $\dim_\mathrm{H} F >1$, then  $D(F)$ has a weak tangent with non-empty interior.
\end{cor}

In the higher dimensional case we obtain partial results.  In particular, we show that the Falconer-Erdo\u{g}an-Wolff bounds for Hausdorff dimension also hold for Assouad dimension (without any measureability assumptions), see \cite{MattilaFourier} for a recent survey of the state of the art.  It appears that the appropriate analogues of these bounds are \emph{not} known to hold for other notions of dimension, such as box-counting or packing dimension.

\begin{thm} \label{distanceset2}
Let $d \geq 2$ be an integer and $F \subseteq \mathbb{R}^d$ be any non-empty set.  Then
\[
\dim_\text{\emph{A}} D( F)  \geq  \inf_{\substack{E  \,  \in  \, \mathcal{K}(  \mathbb{R}^d ) : \\
\dim_{\mathrm{H}}E \, = \,  \dim_{\mathrm{A}} F}} \dim_\text{\emph{A}} D( E).
\]
In particular, if $\dim_\mathrm{A} F < d/2+1/3$, then
\[
\dim_\text{\emph{A}} D( F)  \geq \max  \left\{ \frac{6 \dim_\mathrm{A} F+2-3d}{4}, \, \dim_\mathrm{A} F-\frac{d-1}{2} \right\}
\]
and, if $\dim_\text{\emph{A}} F  \geq d/2+1/3$, then $\dim_\text{\emph{A}} D( F)  = 1.$
\end{thm}

We will prove Theorem \ref{distanceset2} in Section \ref{distanceset2proof}. The  non-quantitative part of Theorem  \ref{distanceset2} provides a recipe for directly transferring \emph{any} bounds concerning the Hausdorff (or box) dimension version of the distance set problem into analogous bounds concerning the  Assouad dimension version.  In particular, this shows that the Hausdorff dimension version of  Conjecture \ref{distanceconj} implies the Assouad dimension version.

Using an elementary `product and project' argument, it is easily seen (and well-known) that $\overline{\dim}_\text{B} D( F) \leq 2 \overline{\dim}_\text{B} F$ for any bounded $F \subseteq \mathbb{R}^d$.   The same bound holds for packing dimension, $\dim_\mathrm{P}$.   However, there exists a compact set  in $\mathbb{R}^d$ (for any $d$) with Hausdorff dimension 0, but whose distance set contains an interval, see \cite{squares, davies}.  In particular, \cite[Proposition 3.1]{squares}  provides a compact set $A$ with Hausdorff dimension $0$ but for which $A \cap (A+t) \neq \emptyset $ for all $t \in [0,1]^d$, for example.  This proves that $D(A)$ contains the interval $[0, \sqrt{d}]$.   Despite these examples, there is still some control on the Hausdorff dimension of the distance sets of small sets in that one always has $\dim_\mathrm{H} D( F) \leq \dim_\mathrm{H} F + \dim_\mathrm{P} F$.  The situation for Assouad dimension is rather different in that arbitrarily small sets from a dimension point of view can have distance sets with full Assouad dimension.

\begin{example} \label{distancesetexamples2}
There exists a non-empty compact set $F \subseteq [0,1]$ with $\dim_\mathrm{A} F  =0$ and   $\dim_\text{\emph{A}} D( F)  = 1.$ 
\end{example}

We will provide the details of this example in Section \ref{distancesetexamples2proof}.  If one allows $F$ to have positive dimension, then the examples can be very simple and are also quite prevalent.   In particular, for every $s \in (0,1)$, there exists an AD-regular set with dimension $s$  but  $\dim_\mathrm{A} D( F)  = 1.$   The following theorem readily yields such examples since self-similar sets satisfying the open set condition are AD-regular. Recall that a set is \emph{self-similar} if it is the unique non-empty compact set $F$ satisfying
\[
F = \bigcup_{i} S_i(F)
\]
for a finite collection of similarity maps $S_i$ whose similarity ratios are all strictly less than 1.  Such a collection of maps is called an iterated function system (IFS), and we refer the reader to \cite[Chapter 9]{falconer} for more details, including the definition of the open set condition.  Roughly speaking, the open set condition is satisfied if the `pieces' $S_i(F)$ do not overlap too much.

\begin{thm} \label{selfsimthm}
Let $F \subseteq \mathbb{R}^d$ be a self-similar set which is not a single point and suppose that two of the defining similarity ratios are given by $a, b \in (0,1)$ satisfying $\log a / \log b \notin \mathbb{Q}$.  Then $\dim_\text{\emph{A}} D( F)  = 1.$ 
\end{thm}

We will prove Theorem \ref{selfsimthm} in Section \ref{distancesetexamples2proof}.  The distance set problem for self-similar sets was considered by Orponen \cite{orpselfsim}.  He proved that if a self-similar set in the plane has positive $\mathcal{H}^1$ measure, then the corresponding  distance set has Hausdorff dimension one. Despite the previous examples, Theorem \ref{distanceset} is still sharp  in the sense that one cannot guarantee that the distance set has full Assouad dimension for sets $F$ with $\dim_\mathrm{A}F>s$ for any $s<1$.  

\begin{example} \label{distancesetexamples1}
For every $s \in [0,1)$, there exists a compact set $F \subseteq [0,1]$ with $\dim_\mathrm{A} F  \geq s$  but   $\dim_\text{\emph{A}} D( F)  < 1.$
\end{example}

We will provide the details of this example in Section \ref{distancesetexamples1proof}.

\subsection{Projections}

How dimension behaves under orthogonal projection is a classical problem in geometric measure theory.  It was first considered by Besicovitch and later by Marstrand; see Marstrand's seminal 1954 paper \cite{Marstrand}.  For integers $k,d$ with $1 \leq k <d$, one considers projections of $\mathbb{R}^d$ onto $k$-dimensional subspaces.  The $k$-dimensional subspaces of $\mathbb{R}^d$ come with a natural $k(d-k)$ dimensional manifold structure, and so come equipped with a $k(d-k)$ dimensional analogue of Lebesgue measure.  We identify each $k$-dimensional subspace $V$ with the orthogonal projection $\pi: \mathbb{R}^d \to V$ and denote the set of all such projections as $G_{d,k}$.  We can thus make statements about almost all orthogonal projections  $\pi \in G_{d,k}$.  The manifold $G_{d,k}$ is usually called the Grassmannian manifold and we refer the reader to \cite[Chapter 3]{mattila} for more information on this manifold and its natural measure. The classical result for Hausdorff dimension, often referred to as Marstrand's projection theorem, is that for a fixed  Borel set $F \subseteq \mathbb{R}^d$,  almost all $\pi \in G_{d,k}$ satisfy
\[
\dim_\text{H}  \pi F \ =  \ \min\{k,\dim_\text{H} F\}.
\]
This was first proved by Marstrand in the case $d=2$ \cite{Marstrand} and later by Mattila \cite{MattilaProj} in the general case.  Similar results exist for upper and lower box dimension and packing dimension in the sense that for almost all $\pi \in G_{d,k}$ the dimension of $\pi F$ is equal to a constant. We refer the reader to the recent survey articles \cite{FalconerFraserJin, MattilaSurvey} for an overview of the rich and interesting theory of dimensions of projections.   A striking difference in the case of Assouad dimension, is that the dimension of $\pi F$ need not be almost surely constant.  In particular, \cite[Theorem 2.5]{FraserOrponen} provided an example of  a set in the plane which projects to sets with two different Assouad dimensions in positively many directions $\pi \in G_{2,1}$.  Our main result on projections shows that one can at least give an almost sure lower bound on the Assouad dimension of $\pi F$.

\begin{thm} \label{orthoproj}
Let $F \subseteq \mathbb{R}^d$ be any non-empty set and $1 \leq k <d$ be an integer.  Then 
\[
\dim_\text{\emph{A}}  \pi F \ \geq  \ \min\{k,\dim_\text{\emph{A}} F\}
\]
for almost all $\pi \in G_{d,k}$.
\end{thm}

We will prove this theorem in Section \ref{orthoprojproof}.  The case when $k=1$ and $d=2$ was proved by Fraser and Orponen, see \cite[Theorem 2.1]{FraserOrponen}, but our proof is completely different.  Recall that Fraser and Orponen proved that one cannot replace the almost sure lower bound with an almost sure equality.  Our result is also sharp: for example, if $F \subseteq \mathbb{R}^d$ is contained in a $k$-dimensional subspace, then it is easy to see that for almost all $\pi \in G_{d,k}$ we have $\dim_\text{A}  \pi F  =  \dim_\text{A} F$.

Interestingly, Theorem \ref{orthoproj} is false if one replaces Assouad dimension by packing or  upper box dimension.  The packing/upper box dimension of $\pi F$ is almost surely constant, but this constant can be strictly smaller than the minimum of $k$ and the packing/upper box dimension of $F$, see \cite{FalconerFraserJin, MattilaSurvey}.

We are also able to estimate the size of the exceptional set in the above theorem, i.e. the zero measure subset of $G_{d,k}$ where the lower bound from the theorem fails.  In general, the exceptional set can be somewhere dense in $G_{d,k}$ and so considering the Assouad dimension of this set is the wrong approach (since the Assouad dimension of any somewhere dense set is equal to the dimension of the ambient space).  We therefore compute the \emph{Hausdorff} dimension of the exceptional set.  Note that $G_{d,k}$ is a smooth manifold of dimension $k(d-k)$ and so our results are formulated to allow comparison with the dimension of the ambient space. 

\begin{thm} \label{orthoprojexceptions}
Let $F \subseteq \mathbb{R}^d$ be any non-empty set and $1 \leq k <d$ be an integer.  For any $s \geq 0$ we have 
\[
\dim_{\mathrm{H}} \left\{ \pi \in G_{d,k} \, : \,  \dim_\text{\emph{A}}  \pi F < s \right\}  \  \leq  \ \sup_{\substack{E  \,  \in  \, \mathcal{K}(  \mathbb{R}^d ) : \\
\dim_{\mathrm{H}}E \, = \,  \dim_{\mathrm{A}} F}} \dim_{\mathrm{H}}  \left\{ \pi \in G_{d,k} \, : \,  \dim_\mathrm{A}  \pi E < s \right\}.
\]
In particular, if  $0 < s \leq   \min\{k,\dim_\text{\emph{A}} F\}$, then 
\[
\dim_{\mathrm{H}} \left\{ \pi \in G_{d,k} \, : \,  \dim_\text{\emph{A}}  \pi F < s \right\}  \  \leq  \  k(d-k) +s-  \max\{k,\dim_\text{\emph{A}} F\}.
\]
\end{thm}

We will prove this theorem in Section \ref{orthoprojexceptionsproof}.  Again, there does not appear to be a direct analogue of this exceptional set result for packing dimension or box dimension.  See \cite{orppacking} for some results in this direction and for an indication of why such analogues are difficult (or even impossible) to obtain.  Similar to Theorem  \ref{distanceset2}, the  non-quantitative part of Theorem  \ref{orthoprojexceptions} provides a recipe for directly transferring \emph{any} bounds concerning the Hausdorff dimension of the exceptional set in the projection theorem for Hausdorff dimension into analogous bounds concerning the projection theorem for Assouad dimension.  For example, the quantitative bounds in Theorem  \ref{orthoprojexceptions} can be improved in the plane by using the non-quantitative part together with \cite[Theorem 1.13]{orpfurst} or \cite[Theorem 4]{bourgainsumprod}. We leave precise formulations to the reader.

Finally, we present some general results concerning restricted families of projections.  Let $P \subseteq \mathbb{R}^n$ be a Borel set with positive $n$-dimensional Lebesgue measure which parameterises a family of projections $\{\pi_t \in G_{d,k} : t \in P\}$ (we only assume that the map $t \mapsto \pi_t$ is a bijection, but in practise it will usually have strong additional regularity properties). One now wants to make statements about the dimensions of $\pi_t F$ for almost all $t \in P$ in situations where $\{ \pi_t : t \in P\}$ is a null set in $G_{d,k}$, i.e. a genuinely \emph{restricted} family of projections where Marstrand's theorem yields no information directly.  There are numerous results along these lines, most focusing on Hausdorff dimension, and we refer the reader to \cite{FalconerFraserJin, FassOrp} for a survey of recent results.  Rather than present several different Assouad dimension analogues, we give one `meta theorem', which can apply in many cases.

\begin{thm} \label{restricted}
Let $F \subseteq \mathbb{R}^d$ be any  non-empty set and $1 \leq k <d$ be an integer. Let $P \subseteq \mathbb{R}^n$ be a positive measure set which parameterises a family of projections $\{\pi_t \in G_{d,k} : t \in P\}$ as above.  Then for almost all $t \in P$ we have
\[
\dim_\text{\emph{A}}  \pi_t  F \ \geq  \ \inf_{\substack{E  \,  \in  \, \mathcal{K}(  \mathbb{R}^d ) : \\
\dim_{\mathrm{H}}E \, = \,  \dim_{\mathrm{A}} F}} \ \underset{s \in P}{\mathrm{essinf}}  \,  \dim_\text{\emph{A}} \pi_s E.
\]
\end{thm}
We will prove this theorem in Section \ref{restrictedproof}.  This result gives a recipe for transforming results on the Hausdorff dimension into results on the Assouad dimension.  To motivate this approach, we give one such example, which follows from a result of F\"assler and Orponen \cite{FassOrp} concerning projections of $\mathbb{R}^3$ onto lines  in a `non-degenerate' family of directions.

\begin{cor} \label{restricted2}
Let $F \subseteq \mathbb{R}^3$ be any  set and $\phi:(0,1) \to S^2$ be a $C^3$ bijection such that for all $t \in (0,1) $  the vectors $\{\phi(t), \phi'(t), \phi''(t)\}$ span $\mathbb{R}^3$ and let $\pi_t$ denote projection onto the line in direction $\phi(t)$. 
\begin{enumerate}
\item If  $\dim_\text{\emph{A}}   F = s >1/2$, then there exists a constant $\sigma(s)>1/2$ such that  for almost all $t \in (0,1)$
\[
\dim_\text{\emph{A}}  \pi_t  F \ \geq  \ \sigma(s).
\]
\item If $\dim_\text{\emph{A}}  F \leq 1/2$, then  for almost all $t \in (0,1)$
\[
\dim_\text{\emph{A}}  \pi_t  F \ \geq  \ \dim_\text{\emph{A}}   F.
\]
\end{enumerate}
\end{cor}

\begin{proof}
This follows immediately from Theorem \ref{restricted} and results from \cite{FassOrp}.  For case {\it 1.}, F\"assler and Orponen proved that if an analytic  set $E \subseteq \mathbb{R}^3$ has Hausdorff dimension $s>1/2$, then the packing dimension (and thus the Assouad dimension) of $ \pi_t  E$ is almost surely bigger than a constant $\sigma(s)>1/2$, see  \cite[Theorem 1.7]{FassOrp}.  For case {\it 2.}, see \cite[Proposition 1.5]{FassOrp}. 
\end{proof}

\section{Remaining proofs}

\subsection{Details of Example \ref{notangentexample}: tangents are not enough} \label{notangentexampleproof}

Let $F \subseteq [0,1]$ be given by
\[
F = \{0\} \ \cup \ \bigcup_{k=1}^\infty \bigcup_{l=0}^k \left\{2^{-k} + l 4^{-k} \right\}.
\]
For each $k$ let $T_k$ be the similarity defined by $T_k(x) = k^{-1}4^k (x -2^{-k})$ and observe that
\[
T_k(F) \cap [0,1] = \bigcup_{l=0}^k \left\{ l /k \right\} \to [0,1]
\]
in $d_{\mathcal{H}}$ as $k \to \infty$.  Therefore $[0,1]$ is a weak tangent to $F$ and we may conclude that $\dim_\mathrm{A} F = 1$. Since $F$ only has one accumulation point (at the origin) we only have to consider tangents at this point.  Indeed, any tangent to $F$ at an isolated point is clearly a singleton and has Assouad dimension equal to zero.   As such, suppose $E$ is a tangent to $F$ at 0.  That is, there exists a sequence of similarity maps $S_n$ $(n \geq 1)$ of the form $S_n = c_n x$ where $1<c_n \nearrow \infty$ as $n \to \infty$  such that
\[
S_n(F) \cap B(0,1) \to E
\]
in $d_{\mathcal{H}}$ as $n \to \infty$.   Here we have assumed without loss of generality that the $S_n$ are orientation preserving, which we may do by considering a subsequence and introducing a reflection if necessary.  Let 
\[
F_0 = \{0\} \ \cup \ \bigcup_{k=1}^\infty \left\{2^{-k}  \right\}
\]
and observe that $\dim_\mathrm{A} F_0 = 0$.  We claim that $E$ is also a tangent to $F_0$, which completes the proof since any (weak) tangent to $F_0$ necessarily has Assouad dimension 0. For each $n$ let
\[
m(n) = \min \{ k \geq 1 \ : \ c_n 2^{-k} \leq 1 \}
\]
and note that $m(n) \to \infty$ as $n \to \infty$.  Observe that
\[
S_n(F) \cap B(0,1) \ \subseteq  \ \{0\} \cup  \bigcup_{k=m(n)}^{\infty}  \left[ c_n2^{-k} , \,   c_n2^{-k} + c_n k 4^{-k} \right]
\]
and
\[
S_n(F_0) \cap B(0,1)  \ = \   \{0\}  \cup  \bigcup_{k=m(n)}^{\infty}  \left\{c_n2^{-k}  \right\}
\]
and therefore
\begin{eqnarray*}
d_\mathcal{H} \Big( S_n(F) \cap B(0,1), \, S_n(F_0) \cap B(0,1) \Big)  & \leq & \sup_{k \geq m(n) } c_n k 4^{-k} \\ 
&=& c_n m(n) 4^{-m(n)}  \\ 
& \leq&  m(n) 2^{-m(n)} \to 0.
\end{eqnarray*}
We conclude that  $S_n(F_0) \cap B(0,1) \to E$ in $d_{\mathcal{H}}$  as $n \to \infty$, as required.

\subsection{Distance sets}

\subsubsection{Proof of Theorem \ref{distanceset}} \label{distancesetproof}

The key technical lemma in proving our results on distance sets is the following.  It states that one can pass questions on distance sets to weak tangents.

\begin{lma} \label{disttangent}
Let $F \subseteq \mathbb{R}^d$ be a non-empty closed set and suppose $E$ is a weak tangent to $F$.  Then
\[
\dim_\mathrm{A} D(F) \geq \dim_\mathrm{A} D(E).
\]
\end{lma}

\begin{proof}
Since $E$ is a weak tangent to $F$, we may find a non-empty compact set $X\subseteq \mathbb{R}^d$ and a sequence of similarity maps $T_k: \mathbb{R}^d \to \mathbb{R}^d$ such that
\begin{equation} \label{basicconvA}
T_k(F) \cap X \to E
\end{equation}
in $d_{\mathcal{H}}$ as $k \to \infty$.  We may clearly assume that $X$ is not a single point and we write $ |X| \in (0,\infty)$ for the diameter of $X$.  Also, for each $k$, write $c_k \in (0,\infty)$ for the similarity ratio of $T_k$.  Consider the sequence of compact sets given by
\[
c_k D(F) \cap [0, |X| ]
\]
where $c_k D(F) = \{c_k z : z \in D(F)\}$ and take a strictly increasing infinite sequence of integers $(k_n)_{n >0}$ such that $c_{k_n} D(F) \cap [0, |X| ]$ converges  in $d_{\mathcal{H}}$ to a compact set $B$.  We may do this since $(\mathcal{K}([0, |X|]), d_\mathcal{H})$ is compact.  In particular, $B$ is a weak tangent to $D(F)$ and so
\[ 
\dim_\mathrm{A} D(F) \geq \dim_\mathrm{A} B.
\]
Thus to complete the proof it suffices to show that $D(E) \subseteq B$.  Let $z = |x-y| \in D(E)$ for some $x,y \in E$.  It follows from (\ref{basicconvA}) that we can find a sequence of pairs  $x_k,y_k \in T_k(F) \cap X$ such that $x_k \to x$ and $y_k \to y$.  For each $k$ we have $T_k^{-1}(x_k),  T_k^{-1}(y_k) \in F$ and so 
\[
c_k^{-1} |x_k-y_k  | =  |T_k^{-1}(x_k) - T_k^{-1}(y_k)  | \in D(F) .
\]
Moreover, $|x_k-y_k  | \leq  |X|$ and therefore
\[
 |x_k-y_k  |  \in c_k D(F) \cap [0, |X|].
\]
It  follows that
\[
z = |x-y| = \lim_{k \to \infty}  |x_{k}-y_{k}  |   = \lim_{n \to \infty}  |x_{k_n}-y_{k_n}  | \in B
\]
which completes the proof.
\end{proof}

We are now ready to prove Theorem \ref{distanceset}.  Let $F \subseteq \mathbb{R}^2$ be a closed set with $\dim_\mathrm{A} F= s  > 1$. We will deal with the non-closed case at the end of the proof.

It follows from Theorem \ref{goodweaktangent} that $F$ has a weak tangent $E$ such that $\mathcal{H}^s(E) >0$.  We may take a compact subset $E' \subseteq E$ with positive and finite $\mathcal{H}^s$ measure, see \cite[Theorem 4.10]{falconer}, and define
\[
\nu = \frac{1}{\mathcal{H}^s(E')} \mathcal{H}^s \vert_{E'}
\]
to be the normalised restriction of $\mathcal{H}^s$ to $E'$.  Thus $\nu$ is a Borel probability measure supported on a compact set  $E'$ of positive $\mathcal{H}^s$ measure. Without loss of generality we may assume that $E'$ is contained in $[0,1]^2$ and note that the $\nu$ measure of the boundary of $[0,1]^2$ (or any square) is zero.  We will now employ the theory of CP-chains, which were introduced by Furstenberg in the seminal paper \cite{Furstenberg} building on his earlier work from the 1960s, see \cite{Furstenberg60s}.  The theory has recently been developed by Hochman \cite{Hochman} and Hochman-Shmerkin \cite{HochmanShmerkin} and has proved a powerful tool in studying many geometric problems.  The idea is to apply ideas from ergodic theory to the measure valued  process generated by zooming in at a point in the support of a fractal measure. 

Let $\mathcal{M}$ denote the collection of all  Borel probability measure supported on $[0,1]^2$, endowed with the topology of weak  convergence.  Let $\mathcal{E}$ be the collection of all half open dyadic boxes contained in $[0,1)^2$ oriented with the coordinate axes.  By half open dyadic box we mean a set of the form $[a,b) \times [c,d)$ where both $[a,b)$ and $[c,d)$ are dyadic intervals of the same length.  For $x \in [0,1)^2$, write $\Delta^k(x)$ to denote the unique $k$th generation box in $\mathcal{E}$ containing $x$ (where `$k$th generation' refers to those boxes of sidelength $2^{-k}$ in the dyadic filtration $\mathcal{E}$). For $B \in \mathcal{E}$, let $T_B : \mathbb{R}^2 \to \mathbb{R}^2$ be the unique rotation and reflection free similarity that maps $B$ onto $[0,1)^d$.  For $\mu \in \mathcal{M}$ and $B \in \mathcal{E}$ such that  $\mu(B) > 0$, we write
\[
\mu^B \ = \ \frac{1}{\mu(B)} \mu|_B \circ T_B^{-1} \ \in \  \mathcal{M}.
\]
The measures $\mu^{\Delta^k(x)}$ are called (dyadic) \emph{minimeasures} (at $x$) and weak limits of sequences of minimeasures where the level $k \to \infty$  are called (dyadic) \emph{micromeasures} (at $x$).  We denote the set of all micromeasures of $\mu$ by $\text{Micro}(\mu)$.  A \emph{CP-chain} is a stationary Markov process $(\mu_n,x_n)_{n=1}^\infty$ on the state space
\[
 \{ (\mu,x) \ : \ \mu \in \mathcal{M} , \, x \in [0,1)^2, \, \text{and for all $k \in \mathbb{N}$, } \mu(\Delta^k(x))>0\}
\]
where the transition probability is given by
\[
(\mu,x) \to \left(\mu^{\Delta^1(x)}, \, T_{\Delta^1(x)}(x)\right)
\]
with probability $\mu(\Delta^1(x))$.  There is a minor technical issue here if $\mu$ gives positive measure to the boundary of the dyadic boxes in $\mathcal{E}$, but we can omit discussion of this since we will apply the theory to $\nu$ which does not have this property. For convenience we assume from now on that $\mu$ gives zero measure to the boundary of all dyadic boxes.   The measure component of the stationary distribution for the  process described above is a measure $Q$ supported on $\mathcal{M}$ and is called a \emph{CP-distribution}.  A CP-chain is said to be ergodic if $Q$ is ergodic. We say a measure $\mu \in \mathcal{M}$ \emph{generates} a CP-chain with measure component $Q$ if at $\mu$ almost every $x \in [0,1]^2$, the scenery distributions
\[
\frac{1}{N} \sum_{k=1}^N \delta_{\mu^{\Delta^k(x)}}
\]
converge weakly to $Q$ as $N \to \infty$ and for every $q \in \mathbb{N}$, there exists a distribution $Q_q$ on $\mathcal{M}$ such that at $\mu$ almost every $x \in [0,1]^2$, the $q$-sparse scenery distributions
\[
\frac{1}{N} \sum_{k=1}^N \delta_{\mu^{\Delta^{qk}(x)}}
\]
converge weakly to $Q_q$ as $N \to \infty$. Here the distributions $Q$ and $Q_q$ are necessarily supported on the micromeasures $\text{Micro}(\mu)$.  We refer the reader to \cite[Section 7]{HochmanShmerkin} for more details on CP-chains. They have proved to be of central importance in several problems on geometric measure theory in the last few years and, in particular, measures which generate ergodic CP-chains have many useful properties.   We will use the following result of Ferguson-Fraser-Sahlsten \cite[Theorem 1.7]{FFS} which relates CP-chains to distance sets. Recall that the (lower) Hausdorff dimension of a measure is defined by $ \dim_\mathrm{H}\mu  = \inf \{  \dim_\mathrm{H} F : \mu(F)>0\}$ and note that $\dim_\mathrm{H} \nu = s$.

\begin{thm}[Ferguson-Fraser-Sahlsten] \label{FFSthm}
Let $\mu \in \mathcal{M}$ be a measure which generates an ergodic CP-chain and is supported on a set $X$ of positive length, i.e. $\mathcal{H}^1(X)>0$.  Then
\[
\dim_\mathrm{H} D(X) \geq \min\{1, \dim_\mathrm{H}\mu \}.
\]
\end{thm}

 The dimension of a CP-chain is the average of the dimensions of micromeasures with respect to the measure component of the chain,  i.e.
\[
\int \dim_{\H} \nu \, d Q(\nu),
\]
but for an ergodic CP-chain the micromeasures are almost surely exact dimensional with a common `exact dimension', \cite[Lemma 7.9]{HochmanShmerkin}. Recall that a measure $\mu$ is called \emph{exact dimensional} if the local dimension 
\[
\lim_{r \to 0} \frac{\log \mu(B(x,r))}{\log r}
\]
exists and equals some constant $\alpha$ at almost every point in the support.   In this case, we also have $ \dim_\mathrm{H}\mu = \alpha$.

Hochman and Shmerkin proved that for \emph{any} $\mu \in \mathcal{M}$, there exists an ergodic CP-chain whose measure component $Q$ is supported on $\text{Micro}(\mu)$ and has dimension at least $\dim_\H \mu$, see \cite[Theorem 7.10]{HochmanShmerkin}. Moreover, \cite[Theorem 7.7]{HochmanShmerkin} tells us that $Q$-almost all micromeasures generate this CP-chain.   In particular, for $\nu$ defined above we can guarantee the existence of a micromeasure $\nu' \in \text{Micro}(\nu)$ which generates an ergodic CP-chain of dimension at least $s$ and satisfies  $\dim_\mathrm{H}\nu'  \geq s > 1$.  This guarantees that the support of $\nu'$ has positive length.     Putting these facts together, there exists a micromeasure $\nu'$ of $\nu$ which generates an ergodic CP-chain and is supported on a set $X$ of positive length.  We now require the following simple general lemma.
\begin{lma} \label{startang}
Let $Z \subseteq \mathbb{R}^d$ be a fixed compact set and $\mu_k$ be a sequence of Borel probability measures with supports denoted by $Y_k \subseteq Z$ such that $\mu_k$ weakly converges to $\mu$ and $Y_k$ converges to $Y$ in the Hausdorff metric.  Then the support of $\mu$ is a subset of $Y$.
\end{lma}

\begin{proof}
For $A,B \in \mathcal{K}(Z)$, let
\[
\rho_\mathcal{H} (A,B) = \inf \{ \delta >0 \ : \   A \subseteq B_\delta \}.
\]
Write $\text{supp}(\mu)$ for the support of $\mu$.  We claim that $\rho_\mathcal{H} (\text{supp}(\mu), Y_k)  \to 0$ as $k \to \infty$.  Suppose not, in which case there exists $\varepsilon>0$ and $x \in \text{supp}(\mu)$ such that there exist arbitrarily large $k$ such that
\[
Y_k \cap B(x,\varepsilon) = \emptyset
\]
where $B(x,\varepsilon)$ denotes the open ball centred at $x$ with radius $\varepsilon$.  Therefore
\[
\mu(B(x,\varepsilon)) \leq \liminf_{k \to \infty} \mu_k(B(x,\varepsilon)) = 0
\]
which is a contradiction. We conclude that
\[
\rho_\mathcal{H} (\text{supp}(\mu), Y) \leq \rho_\mathcal{H} (\text{supp}(\mu), Y_k)  + d_\mathcal{H} (Y_k, Y)  \to 0
\]
and the desired result follows.
\end{proof}

Note that in the above lemma, one may not conclude that the support of $\mu$ is \emph{equal} to $Y$.  The sequence $\mu_k = (1/k) \mathcal{H}^1 \vert_{[0,1]} + (1-1/k) \delta_0$ provides a counter example, where $\delta_0$ denotes a point mass at 0.

Since $\nu'$ is a micromeasure of $\nu$,  it is the weak  limit of a sequence of minimeasures $\nu^{B_k}$ (where $B_k$ is a dyadic square) which are supported on  $T_{B_k}(E') \cap [0,1]^2 \in \mathcal{K}([0,1]^2)$.  Since  $\mathcal{K}([0,1]^2)$ is compact we can assume the sequence $T_{B_k}(E') \cap [0,1]^2 $ converges in $d_\mathcal{H}$ to a non-empty compact set $E'' \subseteq [0,1]^2$, which is therefore a weak tangent to $E'$.  It follows from Lemma \ref{startang} that $ X = \text{supp}(\nu')  \subseteq E''$.

 The desired result now follows by piecing together the above:
\begin{eqnarray*}
\dim_\mathrm{A} D(F) & \geq&  \dim_\mathrm{A} D(E) \qquad \text{by Lemma \ref{disttangent} since $E$ is a weak tangent to $F$} \\ 
& \geq& \dim_\mathrm{A} D(E') \qquad \text{since $E' \subseteq E$} \\
& \geq& \dim_\mathrm{A} D(E'') \qquad \text{by Lemma \ref{disttangent} since $E''$ is a weak tangent to $E'$}  \\
& \geq& \dim_\mathrm{A} D(X )  \qquad \text{since $X \subseteq E''$} \\
& \geq& \dim_\mathrm{H} D(X)  \\
&\geq& \min\{1, \dim_\mathrm{H}\nu' \} \qquad  \text{by Theorem \ref{FFSthm}} \\
&=& 1
\end{eqnarray*}
which completes the proof.

All that remains is the case when $F$ is not closed.  However, $\dim_\mathrm{A}\overline{F} = \dim_\mathrm{A} F >  1$  and $\overline{D(F)} \supseteq D( \overline{F}) $.  Therefore, using the result in the closed case, we have
\[
\dim_\mathrm{A} D(F)  =  \dim_\mathrm{A} \overline{D(F)} \geq  \dim_\mathrm{A}D( \overline{F}) = 1
\]
as required.

\subsubsection{Proof of Theorem \ref{distanceset2}} \label{distanceset2proof}

This theorem follows immediately by combining Theorem \ref{goodweaktangent}, Lemma \ref{disttangent} and the results of Falconer and Erdo\u{g}an  on Hausdorff dimension.  Let $F \subseteq \mathbb{R}^d$ be any non-empty set.  Using  Theorem \ref{goodweaktangent} we can find a compact (and therefore Borel) set $E \subseteq \mathbb{R}^d$ which is a weak tangent to $\overline{F}$ with $ \dim_\mathrm{H} E = \dim_\mathrm{A} F =\dim_\mathrm{A} \overline{F}$.  The desired result then follows from Lemma \ref{disttangent} since
\begin{eqnarray*}
\dim_\mathrm{A} D(F)  =  \dim_\mathrm{A} \overline{D(F)} \geq  \dim_\mathrm{A}D( \overline{F}) \geq \dim_\mathrm{A} D(E)  \geq  \inf_{\substack{E'  \,  \in  \, \mathcal{K}(  \mathbb{R}^d ) : \\
\dim_{\mathrm{H}}E' \, = \,  \dim_{\mathrm{A}} F}} \dim_\text{A} D( E'),
\end{eqnarray*}
as required.  For the quantitative result, the well-known results of Falconer and Erdo\u{g}an, see \cite[Chapter 15]{MattilaFourier}, yield
\[
\dim_\mathrm{A} D(F)   \geq  \dim_\mathrm{A} D( E)  \geq \dim_\mathrm{H} D( E)  \geq \max  \left\{ \frac{6 \dim_\mathrm{A} F+2-3d}{4}, \, \dim_\mathrm{A} F-\frac{d-1}{2} \right\}
\]
provided $\dim_\mathrm{H} E = \dim_\mathrm{A} F < d/2 + 1/3$ and $\dim_\mathrm{H} D( E)  = 1$ otherwise.

\subsubsection{Details of Example \ref{distancesetexamples2} and proof of Theorem \ref{selfsimthm}} \label{distancesetexamples2proof}

We begin by proving Theorem \ref{selfsimthm} and then we will adapt the construction of a self-similar set to satisfy the conditions of Example \ref{distancesetexamples2}.  Let $F \subset \mathbb{R}^d$ be a self-similar set which is not a single point.  Let $x,y \in F$ be distinct points and let $\Delta : = |x-y| >0$.  Let $S_a$ be one of the defining similarity maps with contraction ratio $a$ and $S_b$ be one of the defining similarity maps with contraction ratio $b$.  For all integers $m,n \geq 0$ we have that $S_a^m \circ S_b^n(x), \, S_a^m \circ S_b^n(y) \in F$ and therefore
\[
|S_a^m \circ S_b^n(x) - S_a^m \circ S_b^n(y) | = a^mb^n \Delta \in D(F).
\]
 The proof of Theorem \ref{selfsimthm} is now similar to the example in \cite[Section 3.1]{Fraser}, but we include the details for completeness. We will show that $[0,1]$ is a weak tangent to $D(F)$ which proves  the theorem.  For each integer $k \geq 1 $ let $T_k: [0,1] \to [0,1]$ be defined by $T_k(x) = \Delta^{-1} b^{-k} x$ and, using compactness of $\mathcal{K}([0,1])$, extract a convergent subsequence of $T_k(D(F)) \cap [0,1]$ in the Hausdorff metric, the limit of which is a weak tangent to $D(F)$.  Since 
\[
\big\{a^mb^n    \ : \ m, n \in \mathbb{Z}, \,  m  \geq 0, n  \geq -k\big\} \cap [0,1]  \ \subset \  T_k(D(F)) \cap [0,1]
\]
for all $k \geq 1$ and the sequence of sets on the left is nested, the weak tangent must contain the set 
\[
\overline{\{a^mb^n : m \in \mathbb{N}, n \in \mathbb{Z} \}} \cap [0,1].
\]
However, it follows almost immediately from the assumption on $a$ and $b$ that this set is simply $[0,1]$.  We will prove the equivalent statement that $\{m\log a +n\log b : m \in \mathbb{N}, n \in \mathbb{Z} \}$ is dense in $(-\infty, 0)$. Indeed
\[
m\log a+n\log b = n\log a \bigg(\frac{m}{n}+ \frac{\log b}{\log a}\bigg)
\]
and by Dirichlet's theorem on Diophantine approximation combined with the irrationality of $\log b/\log a$ we can find infinitely many integers $n \geq 1$ such that
\[
0< \Big\lvert \frac{m}{n}+ \frac{\log b}{\log a} \Big\rvert < 1/n^2
\]
for some integer $m$.  Therefore for any $\varepsilon>0$ we can choose $m,n$ such that
\[
0<\lvert m\log a+n\log b \rvert < \varepsilon
\]
and by scaling $m,n$ by each positive integer in turn we can find infinite (one sided) arithmetic progressions with arbitrarily small gap length inside $\{m\log a +n\log b : m \in \mathbb{N}, n \in \mathbb{Z} \}$ which completes the proof.

We will now show how to build an example with Assouad dimension 0 since the examples provided by Theorem \ref{selfsimthm} all have strictly positive Assouad dimension. This is the content of Example  \ref{distancesetexamples2}.  Fix $a,b \in (0,1)$ with $\log a/ \log b \notin \mathbb{Q}$ as above.  For each  integer $k \geq 1$, consider the IFS $\mathcal{I}(k)$ consisting of the two maps
\[
S_k^1 : x \mapsto a^{2^k} x \qquad \text{ and } \qquad S_k^2 :  x \mapsto b^{2^k} x + (1-b^{2^k})
\]
and let $N(k)$ be a large positive integer which we will specify later.  Let $\theta = (\theta_1, \theta_2, \dots) \in \mathbb{N}^\mathbb{N}$ be the infinite integer sequence defined by beginning with $N(1)$ 1s, and following with $N(2)$ 2s, and so on.  In other words, $\theta$ is the unique word with non-decreasing entries such that for all integers $k$ there are precisely  $N(k)$ occurrences of the integer $k$. Also let $\Phi_k : \mathcal{K}([0,1]) \to  \mathcal{K}([0,1]) $ be defined by the action of the IFS $\mathcal{I}(k)$, i.e.
\[
\Phi_k(X) = S_k^1 (X) \cup S_k^2(X).
\]
Finally, let 
\[
E = \bigcap_{n=1}^\infty \Phi_{\theta_n} \circ \cdots \circ \Phi_{\theta_1} ([0,1])
\]
and observe that $E$ is a non-empty compact subset of $[0,1]$.  For every integer $k \geq 1$ we can decompose $E$ into finitely many pieces, each of which is a subset of the attractor of $\mathcal{I}(k)$.  We can do this because the maps in $\mathcal{I}(k')$ are iterates of maps from $\mathcal{I}(k)$ for any $1 \leq k < k'$.  Since the Assouad dimension is finitely stable, this means that we can bound the Assouad dimension of $E$ by the \emph{similarity dimension} of the attractor of $\mathcal{I}(k)$ for all $k$, see \cite[Chapter 9]{falconer} and \cite{Fraser}.  In particular, the similarity dimension associated with $\mathcal{I}(k)$ is the unique real solution $s(k)$ of 
\[
a^{s(k)2^k} + b^{s(k)2^k} = 1
\]
and so
\[
\dim_\mathrm{A} E  \ \leq \  \inf_{k \geq 1} s(k)  \ = \  0. 
\]
All remains is to prove that we may choose the integers $N(k)$ such that the key feature of the systems we use is preserved, i.e. $\dim_\mathrm{A} D(E)=1 $. 

By construction, for all $k\geq 1$ we have
\[
D(E) \ \supseteq \  \left( \prod_{l=1}^{k-1} a^{2^{(l-1)}N(l-1)} \right)  \left\{ a^{2^km} b^{2^kn } \ : \  \ m, n \in \mathbb{Z}, \,  0 \leq m , n \leq N(k)     \right\}.
\]
These points are found similar to above, but by looking at the left most interval at level $N(1)+ \cdots + N(k-1) $ in the construction, which has length 
\[
\left( \prod_{l=1}^{k-1} a^{2^{(l-1)}N(l-1)} \right),
\]
and then looking at end points of intervals within this interval for the next $N(k)$ levels.  For each $k \geq 1$, let $T_k$ be defined by
\[
T_k(x) = \left( \prod_{l=1}^{k-1} a^{2^{(l-1)}N(l-1)} \right)^{-1} b^{-2^kN(k)/2 }(x)
\]
and assume for convenience that $N(k)$ is even. It follows that
\begin{eqnarray*}
 &\, & T_k(D(E)) \cap [0,1]  \\ \\ 
&\,& \qquad  \supseteq \   \left\{ a^{2^km} b^{2^kn } \ : \  \ m, n \in \mathbb{Z}, \,  0 \leq m \leq N(k), \,   -N(k)/2 \leq n \leq N(k)/2    \right\} \cap [0,1].
\end{eqnarray*}
Now choose $N(k)$ sufficiently large (and even) such that the Hausdorff distance between the set above and the closure of the set 
\[
I  \  : = \ \left\{ a^{2^km} b^{2^kn } \ : \  \ m, n \in \mathbb{Z}, \,  0 \leq m < \infty, \,   -\infty <  n < \infty    \right\} \cap [0,1]
\]
is less than $1/k$.  However, since
\[
\frac{\log a^{2^k}}{\log b^{2^k}} = \frac{\log a}{\log b}  \notin \mathbb{Q},
\]
we have already seen that the closure of $I$ is simply the unit interval $[0,1]$.  It follows that $ T_k(D(E)) \cap [0,1]  \to [0,1]$ in the Hausdorff distance as $k \to \infty$.  In particular, $[0,1]$ is a weak tangent to $D(E)$ and we  conclude that $\dim_\mathrm{A} D(E) = 1$ as required.

\subsubsection{Details of Example \ref{distancesetexamples1}: sharpness of Theorem \ref{distanceset} }\label{distancesetexamples1proof}

Let $s \in [0,1)$ and let $N \geq 2$ be an integer satisfying $(N+1)/2 - N^s >1 $ and $K $ be an integer satisfying $N^s \leq K < (N+1)/2$.  For all $i \in \{0, \dots, K-1\}$ let $S_i : [0,1] \to [0,1]$ be defined by
\[
S_i(x) = (x+2i)/N
\] 
and $F \subseteq [0,1]$ be the self-similar set associated with the IFS $\{S_i\}_{i=0}^{K-1}$.  Since the defining IFS satisfies the open set condition, it follows that
\[
\dim_\mathrm{A} F = \dim_\mathrm{H} F= \frac{\log K}{\log N} \geq s.
\]
Consider the set  $\pi(F \times F)$ where $\pi$ is orthogonal projection onto the subspace of $\mathbb{R}^2$ spanned by $(1,-1)$ identified with $\mathbb{R}$.  This is a self-similar set defined by $2K-1$ equicontractive similarities with common contraction ratio $1/N$.  Moreover, by our choice of indexing, a simple geometric argument shows that the open set condition is satisfied for this system. The map $\phi$ defined by $\phi(x) = |x|$ is  bi-Lipschitz on $(-\infty, 0)$ and $[0, \infty)$ and so cannot increase Assouad dimension (the fact that it is Lipschitz on the whole of $\mathbb{R}$ is not enough to guarantee this, see \cite[Example A.6 2]{Luukkainen} or \cite[Section 3.1]{Fraser}). Moreover, $\phi(\pi(F \times F)) = D(F)$ and therefore
\[
\dim_\mathrm{A} D(F)  \leq \dim_\mathrm{A} \pi(F \times F) \leq  \frac{\log 2K-1}{\log N} < 1
\]
as required.

\subsection{Projections}

\subsubsection{Proof of Theorem \ref{orthoproj}: orthogonal projections} \label{orthoprojproof}

Let $F \subseteq \mathbb{R}^d$ be any  non-empty set, and let $k \in [1,d)$ be an integer.  Since for any $\pi \in G_{d,k}$ we have $\pi(\overline{F}) \subseteq \overline{\pi(F)}$ and the Assouad dimension is stable under taking closure, we may assume that $F$ is closed to begin with.  It follows from Theorem \ref{goodweaktangent} that there exists a compact set $E\subseteq \mathbb{R}^d$ which is a weak tangent to $F$ such that $\dim_{\mathrm{H}} E = \dim_{\mathrm{A}} F$.  In particular, there exists a sequence of similarity maps $T_k$ on $\mathbb{R}^d$ and a compact set $X \subseteq \mathbb{R}^d$ such that
\begin{equation} \label{originalconv}
T_k(F) \cap X \to E
\end{equation}
in $d_\mathcal{H}$ as $k \to \infty$.  Moreover, we may assume that the $T_k$ are homothetic, i.e. of the form $T_k(x) = c_k x + t_k$ for a real constant $c_k>0$ and a translation $t_k \in \mathbb{R}^d$.  In fact the proof of Theorem \ref{goodweaktangent} given in \cite{kaenmakiassouad2} yields homothetic maps directly, but it is also easy to prove from the statement of Theorem \ref{goodweaktangent} in this paper.  Suppose $T_k(x) = c_k O_k x + t_k$ where $O_k \in \mathcal{O}(\mathbb{R}, d)$ is a not necessarily trivial orthogonal component and $\mathcal{O}(\mathbb{R}, d)$ is the real orthogonal group.  Since   $\mathcal{O}(\mathbb{R}, d)$ is compact in the topology of uniform convergence, we may assume that $O_k \to O$ uniformly for some fixed $O \in \mathcal{O}(\mathbb{R}, d)$ by taking a subsequence if necessary.  Using continuity of $O^{-1}$ and (\ref{originalconv}) it is easily verified that
\[
(c_k(F)+O^{-1} (t_k)) \cap O^{-1}(X) \to O^{-1}(E)
\]
and therefore $O^{-1}(E)$ is a weak tangent to $F$ with all the desired properties. 

In what follows it is convenient to identify $\pi (\mathbb{R}^d)$ with $\mathbb{R}^k$ in the natural way.  Define a map $\pi \circ T_k \circ \pi^{-1}$ from  $\pi (\mathbb{R}^d)$  to itself by
\[
\{\pi \circ T_k \circ \pi^{-1}(x)\} = \{ \pi(T_k(y)) : \pi(y) = x\}
\]
and observe that since $T_k$ is assumed to be homothetic this is well-defined, i.e. the set $\{ \pi(T_k(y)) : \pi(y) = x\}$ is a singleton. Moreover, writing $T_k(x) = c_k x +t_k$, we have for $x \in \pi (\mathbb{R}^d)$ that
\[
\pi \circ T_k \circ \pi^{-1}(x) = \pi( c_k \pi^{-1}(x) + t_k) =  c_k x + \pi(t_k)
\]
and so $\pi \circ T_k \circ \pi^{-1}$ is itself a (homothetic) similarity.   Since $\pi : \mathcal{K}(\mathbb{R}^d) \to \mathcal{K}(\mathbb{R}^k) $ is continuous, it follows from (\ref{originalconv})  that
\begin{equation} \label{subsetgood}
\left(\pi \circ T_k \circ \pi^{-1}\right) (\pi(F)) \cap \pi(X)  \ = \ \pi (T_k(F)) \cap \pi(X)  \ \supseteq \    \pi \left( T_k F \cap X \right)   \ \to \ \pi(E)
\end{equation}
in $d_\mathcal{H}$ as $k \to \infty$.  Note that $\pi(X)$ is a compact subset of $\pi (\mathbb{R}^d)$ and $\mathcal{K}(\pi (X))$ is compact and so we may assume, by taking a subsequence if necessary, that  $(\pi \circ T_k \circ \pi^{-1}) (\pi(F)) \cap \pi(X)$ converges to a compact set $E' \subseteq \pi(X)$ in $d_\mathcal{H}$ as $k \to \infty$.  In particular, $E'$ is a weak tangent to $\pi(F)$ and it follows from (\ref{subsetgood}) that $E' \supseteq \pi(E)$.

Theorem  \ref{orthoproj} now follows immediately.  We demonstrated above that for \emph{all} $\pi \in G_{d,k}$, the set $\pi(E)$ is a subset of a weak tangent to $\pi(F)$.  It therefore follows from Marstrand's classical projection theorem for Hausdorff dimension that  for \emph{almost all}  $\pi \in G_{d,k}$ we have
\[
\dim_\mathrm{A} \pi(F)  \geq  \dim_\mathrm{A} \pi(E) \geq \dim_\mathrm{H} \pi(E) = \min\{k, \dim_\mathrm{H} E\} = \min\{k, \dim_\mathrm{A} F\}
\]
as required.

\subsubsection{Proof of Theorem \ref{orthoprojexceptions}: dimension of exceptions} \label{orthoprojexceptionsproof}

Theorem \ref{orthoprojexceptions} follows by combining the argument of the previous section with known estimates for the Hausdorff dimension of the set of  exceptions to Marstrand's classical projection theorem for Hausdorff dimension.  In particular, let $E$ be as before and recall that $\dim_{\mathrm{H}} E = \dim_{\mathrm{A}} F$ and for \emph{all} $\pi \in G_{d,k}$ we have $\dim_\mathrm{A} \pi(F)  \geq  \dim_\mathrm{A} \pi(E) \geq  \dim_\mathrm{H} \pi(E)$.   Therefore, for any $s \geq 0$, we have
\begin{eqnarray*}
\dim_{\mathrm{H}}  \left\{ \pi \in G_{d,k} \, : \,  \dim_\mathrm{A}  \pi ( F) < s \right\}  &\leq & \dim_{\mathrm{H}}  \left\{ \pi \in G_{d,k} \, : \,  \dim_\mathrm{A}  \pi (E) < s \right\} \\  
& \leq &  \sup_{\substack{E'  \,  \in  \, \mathcal{K}(  \mathbb{R}^d ) : \\
\dim_{\mathrm{H}}E' \, = \,  \dim_{\mathrm{A}} F}} \dim_{\mathrm{H}}  \left\{ \pi \in G_{d,k} \, : \,  \dim_\mathrm{A}  \pi (E') < s \right\}
\end{eqnarray*}
as required.    Turning attention to the quantitative statement,  suppose $0<s \leq \min\{k, \dim_\mathrm{A} F\} =  \min\{k, \dim_\mathrm{H} E\}$.  Applying  known bounds, which can be found in \cite[Corollary 5.12]{MattilaFourier} for example, yields
\begin{eqnarray*}
\dim_{\mathrm{H}}  \left\{ \pi \in G_{d,k} \, : \,  \dim_\mathrm{A}  \pi ( F) < s \right\}  &\leq & \dim_{\mathrm{H}}  \left\{ \pi \in G_{d,k} \, : \,  \dim_\mathrm{H}  \pi (E) < s \right\} \\ 
& \leq & k(d-k) +s- \max\{k, \dim_\mathrm{H} E\}\\
& = & k(d-k) +s-\max\{k, \dim_\mathrm{A} F \}
\end{eqnarray*}
as required.

\subsubsection{Proof of Theorem \ref{restricted}: restricted families of projections} \label{restrictedproof}

This theorem is proved in a similar way to Theorem \ref{orthoproj}.  Let  $E$ be as before and recall that $\dim_{\mathrm{H}} E = \dim_{\mathrm{A}} F$ and for all $t \in P$ we have $\dim_\mathrm{A} \pi_t(F)  \geq  \dim_\mathrm{A} \pi_t(E)$.  It follows that for almost all $t \in P$ we have
\[
\dim_\mathrm{A} \pi_t(F)  \ \geq   \  \dim_\mathrm{A} \pi_t(E)  \ \geq  \  \underset{s \in P}{\mathrm{essinf}}  \,  \dim_\text{A} \pi_{s}  E \ \geq  \ \inf_{\substack{E  \,  \in  \, \mathcal{K}(  \mathbb{R}^d ) : \\
\dim_{\mathrm{H}}E \, = \,  \dim_{\mathrm{A}} F}} \ \underset{s \in P}{\mathrm{essinf}}  \,  \dim_\mathrm{A} \pi_s E
\]
which completes the proof.

\vspace{6mm}

\begin{centering}

\textbf{Acknowledgements}

The author is  supported by a \emph{Leverhulme Trust Research Fellowship} (RF-2016-500).  He thanks Antti K\"aenm\"aki, John Mackay, and Pablo Shmerkin for helpful discussions.  Finally, he thanks an anonymous referee for helpful comments on the exposition of the article.
\end{centering}

\vspace{5mm}

\noindent \emph{Jonathan M. Fraser\\
School of Mathematics and Statistics\\
The University of St Andrews\\
St Andrews, KY16 9SS, Scotland} \\ \\
\noindent  Email: jmf32@st-andrews.ac.uk

\end{document}